\documentclass[a4paper]{article}
\usepackage{amssymb}% AMS math fonts

\def\pushright#1{{%        set up
   \parfillskip=0pt            % so \par doesnt push \square to left
   \widowpenalty=10000         % so we dont break the page before \square
   \displaywidowpenalty=10000  % ditto
   \finalhyphendemerits=0      % TeXbook exercise 14.32
   \leavevmode                 % \nobreak means lines not pages
   \unskip                     % remove previous space or glue
   \nobreak                    % don't break lines
   \hfil                       % ragged right if we spill over
   \penalty50                  % discouragement to do so
   \hskip.2em                  % ensure some space
   \null                       % anchor following \hfill
   \hfill                      % push \square to right
   {#1}                        % the end-of-proof mark (or whatever)
   \par}}                      % build paragraph

\def\qed{\pushright{$\square$}\penalty-700 \smallskip}

\newenvironment{proof}%
 {\begin{trivlist}\item[]{\bf Proof }}%
 {\qed\end{trivlist}}

\newtheorem{theorem}{Theorem}[section]

\newtheorem{proposition}[theorem]{Proposition}

\newtheorem{lemma}[theorem]{Lemma}

\newtheorem{corollary}[theorem]{Corollary}

\newcommand{\Bb}%blackboard bold
{\mathbb}

\newcommand{\Char}%characteristic set
{{\rm Char}\ }

\newcommand{\stp}[3]% source-target predicate
 {\mbox{$#1\! : #2 \to #3$}}

\newcommand{\Set}% category of sets
 {{\bf Set}}

\newcommand{\Sh}% category of sheaves
 {{\bf Sh}}

\newcommand{\sh}% category of sheaves
 {{\bf sh}}

\newcommand{\type}% type set
{{\rm Type}}

\newcommand{\op}% opposite category
 {^{\rm op}}

\newcommand{\cod}% codomain
 {{\rm cod}}

\newcommand{\Cont}% category of continuous G-sets
 {{\bf Cont}}

\newcommand{\mono}% monomorphism 
 {\rightarrowtail}

\newcommand{\ob}% class of objects
 {{\rm ob}}

\newcommand{\empstg}% empty string
 {[\,]}

\begin{document}

\title{\bf De Morgan's law and the Theory of Fields}
\author{Olivia Caramello and Peter Johnstone\\
DPMMS, University of Cambridge,\\
Wilberforce Road, Cambridge CB3 0WB, U.K.}
\date{\today}
\maketitle

\begin{abstract}
We show that the classifying topos for the theory of fields does not
satisfy De Morgan's law, and we identify its largest dense De Morgan
subtopos as the classifying topos for the theory of fields of nonzero
characteristic which are algebraic over their prime fields.
\end{abstract}

\section*{Introduction}

This note is a tailpiece to a recent paper \cite{caramello} of the
first author, in which necessary and sufficient conditions were given
for the classifying topos of a geometric theory to satisfy De Morgan's
law. A number of examples and counterexamples were given in that
paper; one `test case' which seemed worth considering was the
(coherent) theory of fields, but it turned out that some additional
ideas were needed to handle this case. Interestingly, the germ of
these ideas was present in an old paper \cite{Diers} of the second
author --- which happened to be published in the same volume as the
first paper \cite{dml} in which the topos-theoretic ramifications of
De Morgan's law were explored.

Another new result presented in \cite{caramello} was the fact that,
for every topos $\cal E$, there exists a largest dense subtopos of
$\cal E$ satisfying De Morgan's law; we call this subtopos the {\em
DeMorganization\/} of $\cal E$, by analogy with the Booleanization which
is the largest (in fact only) dense Boolean subtopos. Explicit
examples of DeMorganizations, for toposes which do not satisfy De
Morgan's law, seem to be rather hard to find; but it turns out that
the techniques of this paper give us such a description for the
DeMorganization of the classifying topos for fields, and enable us to
show that it classifies an easily described theory. These
results are presented in section 2 of the paper; section 1 contains the
proof that the classifying topos for fields does not itself satisfy De
Morgan's law.

\section{The Theory of Fields is not De Morgan}

We recall that a commutative ring $R$ is said to be ({\it von Neumann\/}) {\it
regular\/} if, for every $x\in R$, there exists $y\in R$ satisfying $x^2y=x$ 
and $y^2x=y$. Note that this implies that $R$ is nilpotent-free, since from
$x^2=0$ we may deduce $x=x^2y=0$. Also, the $y$ whose existence is asserted is
uniquely determined by $x$, since if $y$ and $z$ both satisfy the equations
we have $x^2(y-z)^2 = (x-x)(y-z)=0$, whence $x(y-z)=0$, and therefore 
$y-z=(y^2-z^2)x = x(y-z)(y+z)=0$. Thus we may think of regular rings as 
commutative rings equipped with an additional unary operation $(-)^*$, 
satisfying $x^2x^*=x$ and $x(x^*)^2=x^*$ for all $x$. (Note that it also 
follows from the uniqueness of $x^*$ that we have $x^{**}=x$.)

Any field becomes a regular ring if we define $x^*=x^{-1}$ for all $x\neq 0$,
and $0^*=0$. Conversely, it is not hard to show that any prime ideal in a
regular ring is maximal, and hence that any regular ring is a subdirect
product of fields.

In what follows, we shall work with the category $\cal C$ of finitely-presented
regular rings, considered as a full subcategory of the category $\bf CRng$ of
commutative rings. (Note that any ring homomorphism between regular rings
automatically commutes with the $(-)^*$ operation; similarly, if $I$
is any (ordinary) ring ideal of a regular ring $R$, the quotient $R/I$
is regular.) However, the reader should
beware that finitely-presented regular rings are not in general 
finitely-presented as rings, because of the presence of the additional 
operation $(-)^*$. 

We may define the notion of characteristic for regular rings, not as a single
prime number but as a set of primes. For definiteness, let us write
${\Bb P}$ for the set of (nonzero) prime numbers, and ${\Bb P}^+$
for ${\Bb P}\cup\{0\}$. Then we define
$$ \Char R = \{p\in{\Bb P}^+\mid {\rm char\ }R/M = p \hbox{ for some maximal
ideal } M\subseteq R\}\ .$$
Equivalently, $p\in \Char R$ iff there exists a homomorphism from $R$ to
{\it some\/} field of characteristic $p$. (For nonzero $p$, we have the
further equivalent condition that $p\in\Char R$ iff $R/(p)$ is non-degenerate.)
We note in passing that if there exists a homomorphism \stp{h}{R}{S},
then $\Char S \subseteq \Char R$; and we have equality here if $h$ is 
injective, since if \stp{k}{R}{F} is a homomorphism from $R$ to a
field, we can find a prime ideal of $S$ disjoint from the set
$$\{h(x)\mid x \in R, k(x)\neq 0\}\ ,$$
and the quotient of $S$ by this ideal must have the same
characteristic as $F$.

\begin{lemma}
If $R$ is a finitely-presented regular ring, then $\Char R \subseteq {\Bb P}^+$
is either a finite set not containing $0$, or a cofinite set containing $0$.
\end{lemma}

\begin{proof}
First suppose there is some $n\in{\Bb Z}^+$ such that $n.1=0$ in $R$. Then
this equation holds in any field to which $R$ can be mapped, so any such field
has characteristic dividing $n$. Hence $\Char R$ is contained in the (finite)
set of prime divisors of $n$.

Otherwise, we have $n.1\neq 0$ for all $n\in{\Bb Z}^+$. Then the elements of
this type form a multiplicatively closed set not containing $0$, so we may find
a prime (hence maximal) ideal $M$ not meeting this set, and the quotient $R/M$ 
is a field of characteristic $0$. Thus $0\in\Char R$. Moreover, since $R$ is 
finitely-generated as a regular ring, $R/M$ must be finitely-generated as a
field extension of $\Bb Q$, so we can write it either as a pure transcendental
extension ${\Bb Q}(t_1,\ldots,t_d)$ or in the form 
${\Bb Q}(t_1,\ldots t_d,\alpha)$ where the $t_i$ are independent
transcendentals and $\alpha$ is algebraic over ${\Bb Q}(t_1,\ldots,t_d)$. We
shall deal with the second case; the first case is easier (as is the
subcase $d=0$ of the second).

We may suppose that the minimal polynomial for $\alpha$ over ${\Bb
  Q}(t_1,\ldots,t_d)$ has the form
$$f_0(t_1,\ldots,t_d)\alpha^n + f_1(t_1,\ldots,t_d)\alpha^{n-1} + \cdots + 
f_{n-1}(t_1,\ldots,t_d)\alpha + f_n(t_1,\ldots,t_d) = 0$$
where the $f_i$ are polynomials in the $t_j$ with integer 
coefficients, and $f_0$ is not identically zero. At the cost, if
necessary, of replacing $\alpha$ by an integer multiple of itself, we
may further suppose that $f_0$ is primitive, i.e.\ that the highest
common factor of its coefficients is $1$. Now let $H$ be the
regular ring generated by $\{x_1,\ldots,x_d,y\}$ subject to the single equation
$$ f_0(x_1,\ldots,x_d)y^n + f_1(x_1,\ldots,x_d)y^{n-1} + \cdots + 
f_n(x_1,\ldots,x_d) = 0\ .$$
We claim that $\Char H = {\Bb P}^+$. It clearly contains $0$, since the 
field $R/M$ occurs as a quotient of $H$. And, for any prime $p$, we
may choose values for $x_1,\ldots,x_d$ in some field extension of
${\Bb Z}/(p)$ such that $f_0(x_1,\ldots,x_d)\neq 0)$, and then choose
a value for $y$ in some algebraic extension of this field which
satisfies the polynomial equation above; so we have a homomorphism
from $H$ to a field of characteristic $p$.

Now, for each finite subset $F$ of $\Bb P$ and each finite subset $G$ of the
set of all primitive polynomials in ${\Bb Z}[x_1,\ldots,x_d]$, let $H_{F,G}$
be the quotient of $H$ obtained by adding the relations $pp^*=1$ for each
$p\in F$, and $g(x_1,\ldots,x_d)g(x_1,\ldots,x_d)^*=1$ for each $g\in G$. 
By an easy extension of the argument above, $\Char H_{F,G} = {\Bb
  P}^+\setminus F$ (note that forcing a primitive polynomial in the
$x_j$ to be invertible does not impose any restrictions on the characteristic).
Moreover, it is clear that the $H_{F,G}$ form a directed diagram in
$\cal C$, since we have a quotient map $H_{F,G}\to H_{F',G'}$ whenever
$F\subseteq F'$ and $G\subseteq G'$, and that the colimit of this
diagram in the category of all regular rings is
isomorphic to $R/M$. But $R$, being finitely-presented, is finitely-presentable
(in the categorical sense) as an object of the latter category; hence the
quotient map $R\to R/M$ factors through $H_{F,G}$ for some pair $(F,G)$.
As we observed earlier, this forces $\Char R\supseteq \Char H_{F,G}$; so
$\Char R$ is cofinite.
\end{proof}

Note in passing that no field of characteristic $0$ can be finitely presented
as a regular ring.

We now impose on ${\cal C}\op$ the Grothendieck topology $J$ which makes
$\Sh({\cal C}\op,J)$ into the classifying topos for the geometric theory
of fields. As described in \cite{elephant}, D3.1.11($b$), this is the smallest
coverage for which the degenerate ring $0$ is covered by the empty cosieve
(we shall tend to think of the covers as cosieves in $\cal C$ rather than
sieves in ${\cal C}\op$) and, for each object $R$ and each $a\in R$, $R$ is 
covered by the cosieve $S^R_a$ generated by the two quotient maps $R\to R/(a)$ 
and $R\to R/(aa^*-1)$. It is not hard to see that, since $aa^*$ is
idempotent, the induced map $R\to R/(a)\times R/(aa^*-1)$ is an isomorphism,
and hence that the coverage $J$ is subcanonical; i.e.\ every representable
functor ${\cal C}\,(S,-)$ is a sheaf for it. We shall also need:

\begin{lemma}
For any $J$-covering cosieve $S$ on an object $R$, we have
$\Char R = \bigcup \{\Char\!\!(\cod\ h)\mid h\in S\}$.
\end{lemma}

\begin{proof}
This is clearly satisfied by any of the generating cosieves $S^R_a$, since each
homomorphism from $R$ to a field factors through one of the two quotient maps
generating it. We may now deduce the result by induction over the construction 
of $J$-covering cosieves; alternatively, we may observe that the cosieves 
satisfying the conclusion of the Lemma themselves form a Grothendieck topology
on ${\cal C}\op$, which must therefore contain $J$.
\end{proof}

Now, for any $A\subseteq {\Bb P}^+$ and any $R\in\ob\ {\cal C}$, let 
$T^R_A$ denote the cosieve of all \stp{h}{R}{S} for which $\Char S \subseteq 
A$. By the previous lemma, $T^R_A$ is a $J$-closed cosieve, so it defines a
subobject of the representable functor ${\cal C}\,(R,-)$ in 
$\Sh({\cal C}\op,J)$. Let us take $R$ to be the initial regular ring $I$,
and consider the cosieves $T^I_A$ and $T^I_B$ where $A$ and $B$ are
complementary subsets of $\Bb P$, both of them infinite.

\begin{lemma}
Considered as subterminal objects in $\Sh({\cal C}\op,J)$, the two 
cosieves just defined satisfy $\neg T_A=T_B$ and $\neg T_B=T_A$. 
\end{lemma}

\begin{proof}
Clearly, $T_A\cap T_B$ contains only the morphism $I\to 0$, since $0$
is the only regular ring whose set of characteristics is empty. But
since $0$ is covered by the empty cosieve, it represents the zero
object of $\Sh({\cal C}\op,J)$. Hence $T_B\subseteq \neg T_A$.
Moreover, a morphism \stp{h}{I}{R} belongs
to $\neg T_A$ iff it is stably disjoint from $T_A$; that is, if the
only morphism $R\to S$ such that the composite $I\to R \to S$
belongs to $T_A$ is the unique morphism $I\to 0$. But this implies
that, for all $p\in A$, the quotient $R/(p)$ must be degenerate; so
the nonzero members of $\Char R$ must all lie in $B$. And since $A$ is
infinite, it follows from Lemma 1.1 that $0\not\in\Char R$; so $\Char
R\subseteq B$ and $h\in T_B$.
\end{proof}

\begin{proposition}
The topos $\Sh({\cal C}\op,J)$ does not satisfy De Morgan's law.
\end{proposition}

\begin{proof}
With the notation of the previous lemma, it suffices to show that 
the join of $T_A$ and $T_B$ in the lattice of $J$-closed subterminal objects
is not the top element ${\cal C}\,(I,-)$; equivalently, that the
union $T_A\cup T_B$ is not $J$-covering. And this follows easily from
Lemma 1.2, since we have $0\in\Char I$, but $0$ does not occur in $\Char
R$ for any $I\to R$ in $T_A\cup T_B$.
\end{proof}

Recall that in \cite{caramello}, Theorem 3.1, a syntactic condition
was given for the classifying topos of a geometric theory $\Bb T$ to
satisfy De Morgan's law. We may illustrate the failure of that
condition explicitly in the present case: let $\phi$ be the formula
(in the empty context) $\bigvee_{p\in A}(p.1=0)$. Then the condition
would require the existence of geometric formulae $\psi_1$ and $\psi_2$
satisfying $(\top\vdash(\psi_1\vee\psi_2))$,
$((\psi_1\wedge\phi)\vdash\bot)$, and such that $\chi\wedge\phi$ is
consistent for every consistent $\chi$ satisfying
$(\chi\vdash\psi_2)$. But the second of these conditions forces
$(\psi_1\vdash\bigvee_{p\in B}(p.1=0))$, since by Lemma 1.3 the
disjunction on the right is the largest geometric formula inconsistent
with $\phi$, and hence the formula
$\psi_2$ must be valid in any field of characteristic $0$. It follows
that there can be only finitely many primes $p$ such that
$(\psi_2\wedge(p.1=0))$ is inconsistent; for if there were infinitely
many such $p$, we could obtain a feld of characteristic $0$ as an
ultraproduct of fields of these characteristics --- and although $\psi_2$
is not necessarily coherent, it can be written as a
disjunction of coherent formulae (\cite{elephant}, D1.3.8), from which
it follows that if $\neg\psi_2$ holds in each factor of an
ultraproduct, it must also hold in the ultraproduct itself.
Hence in particular, the formula $\chi=(\psi_2\wedge(p.1=0))$
must be consistent for some $p\in B$; but then $(\chi\wedge\phi)$
is inconsistent.

Note that the crucial element in the above argument is the fact that
the property of having characteristic zero is not definable by any
geometric formula in the theory of fields. This was already observed
in \cite{Diers}, where it was proved by a topological argument; it
also follows from Lemma 1.1, since if the property were definable by a formula
$\phi$ then the interpretation of $\phi$ in the generic field would be
a $J$-closed cosieve on the initial regular ring $I$, consisting of
morphisms $I\to R$ for which $\Char R = \{0\}$.

\section{The DeMorganization of the Theory of Fields}

In this section, our aim is to identify the DeMorganization of
$\Sh({\cal C}\op,J)$, as defined in \cite{caramello}, 1.3. It is
certainly of the form $\Sh({\cal C}\op,J')$ for some $J'\supseteq J$;
moreover, we can identify at least some of the additional cosieves
which $J'$ must contain. In general, given any subobject $A'\mono A$
in a topos $\cal E$, the inclusion $(\neg A'\vee\neg\neg A')\mono A$ must be
dense for the local operator correponding to the DeMorganization of
$\cal E$, since it can be expressed as the pullback of 
$1\amalg 1\mono\Omega_{\neg\neg}$ along the classifying map of $\neg
A'\mono A$. Hence, by the arguments presented in the previous section,
the $J$-closed cosieve $T_{\Bb P}$ must be $J'$-covering on the
initial regular ring $I$, since it contains $T_A\cup T_B$. But the codomain 
of every morphism 
\stp{h}{I}{R} in this cosieve has only a finite set of
characteristics; and if $\Char R=\{p_1,p_2,\ldots,p_n\}$ then we 
may $J$-cover $R$ by the cosieve generated by the quotient maps 
$R\to R/(p_i)$, $1\leq i \leq n$. (For, by composing covers of 
the form $S^{R'}_{p_i}$, we may show that $R$ is $J$-covered by 
the cosieve generated by these quotients together with $R\to
R/(qq^*-1)$, where $q$ is the product of the $p_i$; but the
latter quotient is degenerate because its characteristic set is empty.)
Hence we see that the cosieve generated by the quotient maps 
$I\to I/(p)$, $p\in{\Bb P}$, is $J'$-covering; equivalently, in
the corresponding quotient theory of fields, the geometric sequent
$$(\top\vdash_{\empstg}\bigvee_{p\in{\Bb P}} (p.1=0)) $$
must be provable. 

Now fix a (nonzero) characteristic $p$, and let ${\Bb I}_p$ denote the
set of all monic polynomials in ${\Bb Z}[X]$ whose coefficients are
all in the range $\{0,1,\ldots,p-1\}$ and which are irreducible as
polynomials over ${\Bb Z}/(p)$. (We include the linear polynomials
$X+j$, $0\leq j\leq p-1$.) We write ${\Bb I}_p^+$ for ${\Bb
  I}_p\cup\infty$; now if $R$ is a regular ring with $\Char R=\{p\}$
and $x\in R$, we define the {\it type\/} of $x$ to be the subset of
${\Bb I}_p^+$ defined by
\begin{itemize}
\item $f(X)\in \type(x)$ iff there exists a homomorphism
  \stp{h}{R}{F}, $F$ a field, such that $f(h(x))=0$; and
\item $\infty\in\type(x)$ iff there exists a homomorphism
  \stp{h}{R}{F}, $F$ a field, such that $h(x)$ is transcendental over
  the prime field of $F$.
\end{itemize}
By arguments like those of Lemma 1.1, we may show that if $R$ is
finitely-presented as a regular ring, then the type of any $x\in R$ is
either a finite set not containing $\infty$, or a cofinite set
containing $\infty$. And if we partition ${\Bb I}_p$ into two
disjoint infinite sets $C$ and $D$, we may construct $J$-closed
cosieves $U_C$ and $U_D$ on $I[x]/(p)$ (the quotient of the free
regular ring on one generator $x$ by the ideal of multiples of $p$),
such that \stp{h}{I[x]/(p)}{R} belongs to $U_C$ iff
$\type(h(x))\subseteq C$, and similarly for $U_D$. Just as before, we
may show that each of $U_C$ and $U_D$ is the negation of the other as
a subobject of ${\cal C}\,(I[x]/(p),-)$,
and that their union is not $J$-covering; but it must be
$J'$-covering. By composing with suitable $J$-covering cosieves as
before, we deduce that $I[x]/(p)$ is $J'$-covered by the cosieve generated
by all quotient maps $I[x]/(p)\to I[x]/(p,f(x))$, $f\in {\Bb I}_p$.

In logical terms, this means that the corresponding theory must
satisfy the sequent
$$((p.1=0) \vdash_x \bigvee_{f\in {\Bb I}_p}(f(x)=0))\ ,$$
and hence that it satisfies the sequent
$$((p.1=0)\vdash_x\bigvee_{f\in{\Bb M}}(f(x)=0)) $$
where $\Bb M$ is the set of all monic polynomials over ${\Bb
  Z}$. Given the sequent established earlier, it follows that we have
$$(\top\vdash_x\bigvee_{f\in{\Bb M}}(f(x)=0))\ ,$$
or equivalently that $I[x]$ itself is covered by the cosieve generated
by all quotient maps $I[x]\to I[x]/(f(x))$, $f\in{\Bb M}$.  

Thus we have shown that $J'$ must contain the coverage ($J''$, say)
which is generated by $J$ together with the cosieves generated by 
$(I\to I/(p)\mid p\in{\Bb P})$ and $(I[x]\to I[x]/(f(x))\mid f\in{\Bb
  M})$. We note that $\Sh({\cal C}\op,J'')$ is dense in $\Sh({\cal
  C}\op,J)$, since these additional cosieves are stably
nonempty. (However, $J''$ is not subcanonical: the induced map $I\to
\prod_{p\in{\Bb P}}I/(p)$ is not an isomorphism (since its domain is
countable but its codomain is not), which says that ${\cal
  C}\,(I[x],-)$ does not satisfy the sheaf axiom for the first of the
new covers.)
To show that $J'$ coincides with $J''$, it suffices to show that
$\Sh({\cal C}\op,J'')$ satisfies De Morgan's law. First we need

\begin{lemma} 
The coverage $J''$ described above is\/ {\em rigid} in the
sense of\/ {\em \cite{elephant}, C2.2.18}: that is, every object $R$ has a
smallest $J''$-covering cosieve, which is generated by the set of morphisms
$R\to S$ for which $S$ is\/ {\em irreducible}, i.e.\ has no covering cosieves 
other than the maximal one.
\end{lemma} 

\begin{proof} 
We note first that every finite field occurs as an object of $\cal C$
(although, as we observed earlier, no field of characteristic zero
does so), and they are irreducible for $J''$ since each $J''$-covering
cosieve
is generated by a family of quotient maps, and a field has no proper
quotients. On the other hand, if $R$ is any finitely-presented regular
ring (with finite generating set $\{x_1,\ldots,x_n\}$, say), then by
pushing out and composing copies of the two new covers we may show
that the cosieve generated by all homomorphisms
\stp{h}{R}{R/(p,f_1(x_1),\ldots,f_n(x_n))}, where $p$ is a prime and each
$f_i$ is a monic polynomial, is $J''$-covering. The codomain ($S$,
say) of such a morphism need not be a field; but it is a finite
regular ring. So, if it is not already a field (or degenerate), it
contains some nontrivial idempotent $yy^*$ (where $y$ is neither
invertible nor zero), and the cosieve generated by the quotient maps
$S\to S/(yy^*)$ and $S\to S/(yy^*-1)$ is $J$-covering. Proceeding
inductively, we deduce that $S$ is $J$-covered by the cosieve
generated by its quotient maps to fields, and hence $R$ is
$J''$-covered  by the cosieve generated by its quotient maps to fields.
\end{proof}

\begin{corollary}
$\Sh({\cal C}\op,J'')$ satisfies De Morgan's law.
\end{corollary}

\begin{proof}
The previous lemma implies that it is equivalent to the functor
category $[{\cal C}',\Set]$, where ${\cal C}'$ is the full subcategory 
of $\cal C$ whose objects are finite fields; and this category
satisfies the Ore condition (cf.~\cite{elephant}, D4.6.3($a$)).
\end{proof}

Thus we have proved

\begin{proposition}
The DeMorganization of the classifying topos for the theory of fields is
the classifying topos for the geometric theory of fields of finite
characteristic, in which every element is algebraic over the prime
field.
\qed\end{proposition}

To complete the story, we may also identify the Booleanization of
$\Sh({\cal C}\op,J)$ --- equivalently, of $\Sh({\cal C}\op,J'')$. Note
first that since ${\cal C}'$ is the disjoint union of its
subcategories corresponding to the different primes, $\Sh({\cal
  C}\op,J'')\simeq[{\cal C}',\Set]$ is a coproduct of open
subtoposes corresponding to the primes, and its Booleanization is the
coproduct of the Booleanizations of these subtoposes. So it
suffices to fix a prime $p$, and consider the Booleanization of the
topos $[{\cal C}'_p,\Set]$ where ${\cal C}'_p$ denotes the category of
finite fields of characteristic $p$. Since this category satisfies the
Ore condition, its double-negation coverage simply consists of all
nonempty cosieves; that is, each finite field extension generates a
covering cosieve, This means that the corresponding theory satisfies
the sequents which say that every nonzero polynomial has a root; hence
the Booleanization of $[{\cal C}'_p,\Set]$ classifies the theory of
algebraically closed fields of characteristic $p$ which are algebraic
over ${\Bb Z}/(p)$. By \cite{elephant}, C3.5.8, this topos is atomic
(though not coherent; of course, coherence fails because of the
infinite disjunction in the axiom 
saying that every element is algebraic). Hence, by arguments like
those in \cite{elephant}, D3.4.10, it may be
identified with the topos $\Cont(G_p)$ of continuous $G_p$-sets, where
$G_p$ is the group of automorphisms of the algebraic closure of ${\Bb
  Z}/(p)$ (topologized, as usual, by saying that the pointwise stabilizers of
finite subsets form a basis for the open subgroups). Thus we have

\begin{proposition}
The Booleanization of the classifying topos for fields is the
classifying topos for the theory of algebraically closed fields of
finite characteristic, in which every element is algebraic over the
prime field. Moreover, it is atomic over $\Set$; in fact it may be
identified with the coproduct, over all primes $p$, of the
toposes $\Cont(G_p)$.
\qed\end{proposition} 

It is of interest to note that the (coherent) theory of algebraically
closed fields of a given characteristic is complete but not
$\aleph_0$-categorical (because a countable algebraically closed field
may or may not contain transcendentals); hence, by the results of
Blass and \v{S}\v{c}edrov \cite{bs}, the classifying topos of this theory is
not Boolean. If we add the infinitary geometric axiom which says that
there are no transcendentals, we get a theory which is
`$\aleph_0$-categorical' (to the extent that this concept is
meaningful for infinitary theories), and its classifying topos is not
merely Boolean but atomic. It is also worth noting the following:

\begin{corollary}
Let $\cal E$ be the classifying topos for the theory of algebraically
closed fields of some fixed characteristic $p$. Then the DeMorganization
of $\cal E$ coincides with its Booleanization.
\end{corollary}

\begin{proof}
First note that we may obtain a classifying topos for the theory of
fields of characteristic $p$ by cutting down the site $({\cal
  C}\op,J)$ to the full subcategory ${\cal C}_p$ of finitely-presented regular
rings in which $p.1=0$ (and leaving the definition of the coverage
unchanged). It is readily seen that the classifying topos for
algebraically closed fields of characteristic $p$ is a dense subtopos
of this; hence, by \cite{caramello}, Proposition 1.5, its
DeMorganization is simply
its intersection with the DeMorganization of $\Sh({\cal C}_p\op,J)$. But
the latter classifies the theory of fields of characteristic $p$ in
which every element is algebraic; so when we form the intersection we
obtain the Booleanization of $\Sh({\cal C}_p\op,J)$.
\end{proof}

The result of the Corollary remains true in characteristic $0$, but
the proof requires a little more work. We omit the details.

Finally, we make some remarks about the comparison between the
(coherent) theory of fields and
the theory considered in \cite{Diers} under the name `Diers
fields'. The latter is the theory classified by the functor category
$[{\cal D},\Set]$ where $\cal D$ is the category of all fields which
are finitely-generated over their prime fields (the
finitely-presentable objects in the category of fields). Since this
category also satisfies the Ore condition, the classifying topos for
Diers fields satisfies De Morgan's law. In \cite{Diers} the second
author gave a sketch of how to present the theory of Diers
fields by adding appropriate new predicates and axioms to the coherent
theory of fields, from which it follows that there is a canonical
geometric morphism \stp{u}{[{\cal D},\Set]}{\Sh({\cal C}\op,J)}
corresponding to the forgetful functor from Diers fields to
fields. This morphism is surjective (since every morphism $\Set
\to\Sh({\cal C}\op,J)$ factors through it); so it might be tempting to
conjecture that it is the Gleason cover of $\Sh({\cal C}\op,J)$ in the
sense of \cite{elephant}, D4.6.8. However, this is not the case; since
$\Sh({\cal C}\op,J)$ is a coherent topos and $[{\cal D},\Set]$ is
not compact (\cite{Diers}, 5.1), the geometric morphism between them
cannot be proper (\cite{elephant}, C3.2.16(i)). Also, we have

\begin{lemma}
The geometric morphism $u$ defined above is not skeletal in the sense of\/
{\em \cite{elephant}, D4.6.9}.
\end{lemma}

\begin{proof}
We recall that a morphism \stp{f}{{\cal F}}{{\cal E}} is skeletal iff
it maps $\sh_{\neg\neg}({\cal F})$ into $\sh_{\neg\neg}({\cal E})$. 
But we may identify the Booleanization of $[{\cal D},\Set]$, as we did
that of $\Sh({\cal C}\op,J)$: since every field extension in $\cal D$
(including transcendental extensions) generates a covering sieve for
the corresponding Grothendieck topology, it is easy to see that the
points of $\sh_{\neg\neg}([{\cal D},\Set])$ are the fields which are
injective with respect to all morphisms of $\cal D$; hence they are
exactly the
algebraically closed fields (of any characteristic) which have
infinite transcendence degree over their prime fields. (It is
straightforward to use the additional predicates in the theory of
Diers fields to express in geometric terms the statement that a field
contains infinitely many independent transcendentals.) So $u$ does not
map the points of $\sh_{\neg\neg}([{\cal D},\Set])$ into those of
$\sh_{\neg\neg}(\Sh({\cal C}\op,J))$.
\end{proof}

We do not know whether $u$ can be factored through the Gleason
cover of $\Sh({\cal C}\op,J)$. Since it is not skeletal, we cannot
appeal to \cite{elephant}, D4.6.12 to construct such a factorization;
and since $[{\cal D},\Set]$ is not localic over $\Set$, we cannot
appeal to the projectivity theorem (\cite{elephant}, D4.6.15) either.

\end{document}